\documentclass[12pt,pdftex]{amsart} 

\usepackage{0header}
\allowdisplaybreaks[1]

\title{The generalized Cartan-Hadamard conjecture}

\author[M. Agnoletto]{M. Agnoletto}
\address{\textsc{Marcos Agnoletto}\newline
\indent CMCC, Federal University of ABC, Santo André, SP, Brazil}
\email{marcos.forte@ufabc.edu.br}

\author[M. F. Da Silva]{M. F. Da Silva}
\address{\textsc{M\'{a}rcio Fabiano Da Silva}\newline
\indent CMCC, Federal University of ABC, Santo André, SP, Brazil}
\email{marcio.silva@ufabc.edu.br.}

\author[S. Nardulli]{S. Nardulli}
\address{\textsc{Stefano Nardulli}\newline
\indent CMCC, Federal University of ABC, Santo André, SP, Brazil}
\email{stefano.nardulli@ufabc.edu.br}

\author[R. Resende]{R. Resende}
\address{\textsc{Reinaldo Resende}\newline
\indent Florida State University}
\email{rresende@fsu.edu}

 \date{}
\begin{document}

\date{}

\begin{abstract}
We establish the generalized Cartan-Hadamard conjecture, including rigidity in the equality case. As a corollary, we obtain that isoperimetric regions are bounded. We prove the conjecture for Alexandrov spaces with two sided curvature bounds as well. A sharp Michael--Simon-type inequality for hypersurfaces in these spaces is also shown.
\end{abstract}

\maketitle

\setcounter{tocdepth}{1}
{  \hypersetup{linkcolor=black}
  \tableofcontents
}

\section{Introduction}

This paper continues our work \cite{ADNRdiff} on the generalized Cartan--Hadamard conjecture. There we established the conjecture for small volumes under a lower Ricci curvature bound and reduced the arbitrary volume problem to proving the rigidity statement in a suitable class of Alexandrov spaces; see \cite[Assumption~1 and Theorem~1.3]{ADNRdiff}. Here we establish the rigidity statement needed in that reduction, obtaining the generalized conjecture, including its equality case, in every dimension. We refer to our work \cite{ADNRdiff} or Ghomi and Spruck's work \cite{GhomiSpruck2022} for a more extensive account of the history of the problem and of the compactness and reduction arguments underlying our approach.  

Recall that a Cartan-Hadamard manifold is a complete, simply connected Riemannian manifold with $\Sec_g\leq 0$. We denote its isoperimetric profile by $\ip{M^n}{g}$ and that of the simply connected space form $\mathbb{M}^n_{\bar k}$ of constant sectional curvature $\bar k$ by $\ipm{\bar k}$.

\begin{theorem}\label{thm:smooth CH}
Let $(M^n,g)$ be a Cartan-Hadamard manifold with $\Sec_g \leq \bar k \leq 0$. Then 
\begin{equation*}
    \ip{M^n}{g}(v)\geq \ipm{\bar k}(v)\text{ for any }v\geq 0.
\end{equation*}
Moreover, if there exists a finite perimeter set $E\subset M$ with finite volume $v > 0$ and $P_g(E) = \ipm{\bar k}(v)$, then $E$ is isometric to a geodesic ball in $\mathbb{M}^n_{\bar k}$.
\end{theorem}
\begin{remark}
As a consequence of Theorem~\ref{thm:smooth CH}, by \cite[Corollary 1.6]{ADNRdiff}, any isoperimetric region in Cartan-Hadamard manifolds is bounded.
\end{remark}
\begin{remark}
As shown in \cite[Proposition~8.1]{ADNRdiff}, these techniques extend to Alexandrov spaces with two-sided curvature bounds, and therefore the conjecture holds true in those spaces as well, see Theorem~\ref{thm:main}.
\end{remark}

For the Euclidean comparison, corresponding to $\bar k=0$, the classical results are due to Weil in dimension $n=2$  \cite{Weil1926}, with a subsequent proof by Beckenbach and Rad\'o \cite{BeckenbachRado1933}, Kleiner in dimension $n=3$ \cite{Kleiner1992}, and Croke in dimension $n=4$ \cite{Croke1984}. More recently, Chen, Ghomi, and Wang proved the conjecture in dimension $n=5$ \cite{chen2026cartanhadamardconjecturedimension} and subsequently in dimensions $3\leq n\leq 9$ \cite{chen2026isoperimetricinequalitycmchypersurfaces}. For the generalized conjecture, the classical results are due to Bol in dimension $n=2$ \cite{Bol1941} and Kleiner in dimension $n=3$ \cite{Kleiner1992}. 

The recent paper by OpenAI \cite{OpenAI2026} establishes the inequality in the generalized conjecture in every dimension $n\geq 2$ \cite[Theorem~1.1]{OpenAI2026} and classifies bounded equality regions for the Euclidean comparison \cite[Theorem~1.2]{OpenAI2026}. Its rigidity theorem does not address the case $\bar k<0$. The present work treats equality
for every $\bar k\leq 0$ and does so in the Alexandrov setting required by our earlier reduction in \cite{ADNRdiff}.

\subsection{The sharp Sobolev inequality}

It is well-known that Sobolev inequalities provide an analytic formulation of isoperimetric estimates. For submanifolds, the corresponding inequalities include a mean-curvature term. The foundational results of Allard \cite{Allard1972} and Michael and Simon \cite{MichaelSimon1973} give such inequalities for submanifolds (more generally, for varifolds) of Euclidean space. Hoffman and Spruck \cite{HoffmanSpruck1974} extended this framework to Riemannian ambient manifolds, with hypothesis on the ambient geometry.

Finding the sharp constants on such inequalities requires additional ideas, in Euclidean space, Aubin \cite{Aubin1975} and Talenti \cite{Talenti1976} proved sharp Sobolev inequalities. Castillon \cite{Castillon2010} improved the constant in Michael and Simon's \cite{MichaelSimon1973} and Hoffman and Spruck's \cite{HoffmanSpruck1974} results, but did not show the optimal one. Brendle \cite{Brendle2021} subsequently proved the Sobolev inequality with the sharp Euclidean constant for \emph{minimal} submanifolds in codimension at most two. Brendle and Eichmair \cite{BrendleEichmair2023} gave an alternative proof using optimal transport. Brendle \cite{Brendle2023} established Michael-Simon-type inequalities for submanifolds of manifolds with \emph{nonnegative} sectional curvature. 

The inequality used here is a sharp quadratic inequality at the critical Sobolev exponent. It should be distinguished from the classical $L^1$ inequality: the usual passage from an $L^1$ estimate to an $L^2$ estimate does not by itself yield the specific sharp coefficients required below. Case and Mengesha \cite{CaseMengesha2026} establish the sharp $L^2$-Michael--Simon inequality in Euclidean space in arbitrary codimension and characterize equality. In particular, \cite[Corollary~1.4]{CaseMengesha2026} gives precisely \eqref{eq:sob} when the ambient space is a space form. 

The smooth Cartan-Hadamard version of the sharp Michael-Simon inequality is established by OpenAI in \cite[Theorem~3.1]{OpenAI2026}. We provide a proof in the Alexandrov setting in the Appendix~\ref{app: MS}.

\begin{theorem}[Michael-Simon inequality]\label{prop:sob}
Let $n\geq 4$, $(Z^n,d)$ is a non-collapsing Alexandrov space without Alexandrov boundary and with curvature bounded from above by $\bar k \leq 0$ and from below by $\underline k$, and ${\Sigma}^{n-1}\subset Z$ be an embedded oriented $C^2$ hypersurface without boundary. Then, for every compactly supported Lipschitz function $v$ on ${\Sigma}$, 
\begin{equation}\label{eq:sob}
    \mathcal{Q}(v):=\int_{\Sigma}\bigl((q-2)|\nabla_{\Sigma} v|^2 +(n-1)(H_{\Sigma}^2+\bar k)v^2\bigr)\,d\cH^{n-1}_d \ge c_n\|v\|_{L^q({\Sigma})}^2,
\end{equation}
where $q:=\frac{2(n-1)}{n-3}$, $c_n := (n-1)\omega_{n-1}^{2/(n-1)}$ and $H_{\Sigma}$ is the normalized mean curvature of $\Sigma$. The same holds for the zero extension of every $v\in W^{1,2}_0(U)$, where $U\Subset{\Sigma}$ is relatively open in $\Sigma$.
\end{theorem}

\subsection{Contribution and relation to earlier work}

Our earlier paper \cite{ADNRdiff} was first submitted to the arXiv on September~15, 2026, before the public release of OpenAI's papers on October~6, 2026. OpenAI's manuscript \cite{OpenAI2026}, dated September~23, 2026, is number 337 of that collection.

The present argument combines the reduction developed in \cite{ADNRdiff} with analytic ingredients adapted from \cite{OpenAI2026}. More precisely, Proposition~\ref{prop: equality case+first failure point} establishes the equality statement required at a first possible failure volume in our argument in \cite[Proposition~8.1]{ADNRdiff}.

Thus our contribution in this manuscript is to establish the required rigidity for all $\bar k \leq 0$. While the sharp Michael-Simon inequality and balancing maps come from \cite{OpenAI2026}, their use in the equality argument allows us to complete the program of \cite{ADNRdiff}. We include proofs of all analytic ingredients adapted from \cite{OpenAI2026}, so that the present paper and \cite{ADNRdiff} together provide a complete proof of the generalized Cartan-Hadamard conjecture—including the rigidity part—without requiring those ingredients as external black boxes.

\subsection{AI disclosure}

The mathematical work in \cite{ADNRdiff} was developed without the use of AI, apart from assistance with English grammar and flow. During the preparation of the present article, the authors used Claude Fable~5.1 and ChatGPT-6 Astra to assist in understanding OpenAI's paper \cite{OpenAI2026}. In this process, we realized that using the techniques from \cite{OpenAI2026} would allow us to finish our program initiated in \cite{ADNRdiff}, i.e., to remove \cite[Assumption 1]{ADNRdiff} from \cite[Theorem 1.3]{ADNRdiff} and then establish the full generalized conjecture. After that, we used both AI models to aid in the technical work, computations, and reference searches. We drew the pictures by hand, and ChatGPT transformed them into the TikZ figures presented below.

There is no AI-written text in this article, and the authors take full responsibility for this manuscript.

\subsection{Acknowledgments} 

M. Agnoletto acknowledges financial support from CAPES Grant No.~88887.667684/2022-00, CNPq Grants No.~201543/2024-9 and 201742/2025-0, and FAPESP Grant No.~2021/05256-0.

S. Nardulli was supported by FAPESP Young Investigator Grant No.~2021/05256-0, CNPq Research Productivity Fellowship A No.~309345/2026-0 and 1D No.~12327/2021-8, and CNPq Grants No.~441922/2023-6, 409669/2025-3, and No.~200817/2025-6. 

R. Resende gratefully acknowledges support from FAPESP Grant No.~2021/05256-0 and CNPq Grants No.~409669/2025-3 and ~441922/2023-6 for research visits to UFABC during which part of this work was carried out.

\section{Preliminaries}

\begin{center}\fbox{We use the same notations and definitions of \cite{ADNRdiff}. }\end{center}

Throughout the paper, $(Z,d)$ is an $n$-dimensional non-collapsing Alexandrov space without Alexandrov boundary and with curvature bounded from above by $\bar k \leq 0$ and from below by $\underline k$.
By \cite[Proposition~2.2(iii)]{ADNRdiff}, $Z$ admits a $C^{3,\alpha}$ atlas and a $C^{1,\alpha}$ Riemannian metric $g$ whose induced distance is $d$, for every $\alpha\in(0,1)$.
We work in this atlas throughout.

For $p\in Z$, write $r_p:=d(p,\cdot)$. By \cite[Propositions~1.6 and~4.1]{KL}, $r_p\in C^{1,1}_{\mathrm{loc}}(Z\setminus\{p\})$ and $r_p^2\in C^{1,1}_{\mathrm{loc}}(Z)$. The global curvature condition ($\mathrm{CAT}(\bar k)$ and $\mathrm{CBB}(\underline k)$) give
\begin{equation}\label{eq:hessr}
\frac{s_{\bar k}'(r_p)}{s_{\bar k}(r_p)}
(g-dr_p\otimes dr_p)
\le\Hess_g r_p
\le\frac{s_{\underline k}'(r_p)}{s_{\underline k}(r_p)}
(g-dr_p\otimes dr_p)
\end{equation}
on $Z\setminus\{p\}$, almost everywhere and weakly. This follows from the comparison argument of \cite[Section~2.4]{KL} and \cite[Theorems~8.23 and~9.25]{AKP}. Moreover, equation~(2.2) in the proof of \cite[Proposition~2.2]{ADNRdiff} gives, almost everywhere and weakly on $Z$, the following
\begin{equation}\label{eq:hessdistance}
\Hess_g\left(\frac{r_p^2}{2}\right)\ge g
\end{equation}

For an oriented hypersurface $\Sigma^{n-1}$ with unit normal $\nu_\Sigma$, we choose the following convention for the mean curvature $H_\Sigma:=(n-1)^{-1}\operatorname{div}_\Sigma \nu_\Sigma$. The following lemma is standard for $C^2$ functions and, by a mollification argument in Appendices \ref{app:lem restriction item i} and \ref{app:lem restriction item ii}, we extend it to $C^{1,1}$ functions.

\begin{lemma}\label{lem:restriction}
Let $ u \in C^{1,1}_{\mathrm{loc}}(Z)$ and let $B$ be a continuous symmetric covariant two-tensor. If $\Hess_g u \le B$ weakly, then 
\begin{enumerate}[\upshape (i)]
    \item on every embedded oriented $C^2$ hypersurface $\Sigma^{n-1}$, it holds in the weak sense that
    \begin{equation}\label{eq:restrict}
    \Delta_\Sigma u 
    \le\operatorname{tr}_{T\Sigma}B-(n-1)H_\Sigma\,\partial_{\nu_\Sigma}u ;
    \end{equation}

    \item for every compact set $K\subset Z$, $\alpha\in (0,1)$ and $p\geq 1$, {there exist a relatively compact neighborhood $U$ of $K$, $u\in C^{3,\alpha}(\overline U)$, and numbers $\delta_\varepsilon\searrow 0$ such that
    \begin{equation}\label{eq:ambient-approximation}
     u _\varepsilon\longrightarrow u \quad\text{in }C^1(\overline{U})\text{ and in }W^{2,p}(U),\,\,\,\text{and}\,\,\,
    \Hess_g u _\varepsilon\le B+\delta_\varepsilon g\,\,\text{on}\,\,K.
    \end{equation}}
\end{enumerate}
Moreover, if $\Hess_g u \ge B$ weakly, the inequalities in \eqref{eq:restrict} and \eqref{eq:ambient-approximation} are reversed.
\end{lemma}

\subsection{The {balancing} map}

By \cite[Proposition~8.3]{berestovskij1993multidimensional}, $Z$ is geodesically complete, cf. also \cite[Section~2.2]{KL}. Completeness and the $\mathrm{CAT}(0)$ property {imply that the exponential map is }a homeomorphism $\exp_p:T_pZ\to Z$ for every $p\in Z$. 
Identifying $T_pZ$ isometrically with the tangent space at the origin of $\mathbb M_{\bar k}^n$, define
\begin{equation}\label{eq:radial-comparison-map}
\Psi_p:=\exp_0^{\mathbb{M}^n_{\bar k}}\circ\exp_p^{-1}:Z\longrightarrow\mathbb M_{\bar k}^n.
\end{equation}
By a direct computation using the upper and lower curvature comparisons (namely \cite[Equations (4.2) and (4.3)]{ADNRdiff} that follows from \eqref{eq:hessr} in our setting of Alexandrov spaces) imply that $\Psi_p$ is $1$-Lipschitz and has locally Lipschitz inverse.

Fix $p\in Z$, $\lambda>0$. For $x\ne p$, let $\theta_p(x)=\exp_p^{-1}(x)/r_p(x)$ and define the locally Lipschitz map $\Phi_{p,\lambda}:Z \to \mathbb{S}^n\subset T_pY\oplus \R$ by
\begin{equation}\label{eq:Phi}
\Phi_{p,\lambda}(x):=
\left(\frac{2z}{1+z^2}\theta_p(x),\frac{1-z^2}{1+z^2}\right),
\
z:=\lambda\frac{s_{\bar k}(r_p(x)/2)}{s_{\bar k}'(r_p(x)/2)}, \text{ and }\Phi_{p,\lambda}(p):=(0,1).
\end{equation}
See Figure~\ref{fig:balancing-construction} for intuition on this map. Also define $f_{p,\lambda} \in C^{1,1}_{loc}(Z)$ by
\begin{equation*}
f_{p,\lambda}(x):=\frac{\lambda}{s_{\bar k}'(r_p(x)/2)^2+\lambda^2s_{\bar k}(r_p(x)/2)^2}\\=\frac{\lambda}{1+(\lambda^2-\bar k)s_{\bar k}(r_p(x)/2)^2}.
\end{equation*}
Whenever $p$ and $\lambda$ are fixed, we omit from the notation. 
% Figure 2: geodesic spheres and their spherical images for bar k = 0.
% The normalization in main.tex gives the equatorial radius 2/lambda.
% Complete figure; the source outline and sphere match Figure 3.
\begin{figure}[!htbp]
\centering
\setlength{\abovecaptionskip}{6pt}
\resizebox{.98\linewidth}{!}{%
\begin{tikzpicture}[aubin base]
  \ifdefined\aubinFullPage
    \path[use as bounding box] (.75,.90) rectangle (16.50,5.95);
  \fi

  % The same source patch and base point as in Figure 3.
  \begin{scope}[yshift=.65cm]
    \def\sourcepatch{
      (1.00,.81) .. controls (.72,.87) and (1.07,2.62) .. (2.24,3.69)
      .. controls (3.38,4.36) and (6.28,3.82) .. (7.35,3.04)
      .. controls (6.60,2.56) and (6.16,1.97) .. (5.90,1.46)
      .. controls (4.39,1.89) and (2.00,1.23) .. (1.00,.81) -- cycle}
    \coordinate (p) at (3.84,2.65);
    % A gently curved schematic geodesic sphere, in coordinates about p.
    \def\geodesicsphere{
      (-1.20,-.07) .. controls (-1.13,.24) and (-.62,.57) .. (-.08,.55)
      .. controls (.39,.57) and (1.03,.32) .. (1.20,.05)
      .. controls (1.36,-.19) and (.72,-.43) .. (.20,-.46)
      .. controls (-.36,-.51) and (-1.20,-.34) .. (-1.20,-.07) -- cycle}
    \begin{scope}[overlay]
      \clip \sourcepatch;
      \begin{scope}[shift={(p)}]
        \foreach \factor in {1.28,1.62,1.96}{
          \draw[draw=aubinGreen,line width=.70pt,scale=\factor] \geodesicsphere;
        }
        \path[fill=aubinRed!3] \geodesicsphere;
        \foreach \factor in {.30,.58,.86}{
          \draw[draw=aubinRed,line width=.70pt,scale=\factor] \geodesicsphere;
        }
        \draw[draw=aubinPurple,line width=1.20pt] \geodesicsphere;
      \end{scope}
    \end{scope}
    \draw[aubin outline] \sourcepatch;
    \node[aubin label,fill=white,inner sep=1pt] at (1.65,1.55) {$Z$};
    \node[aubin label,text=aubinPurple,anchor=west,fill=white,inner sep=1.2pt]
      (ball-label) at (7.15,4.05) {$B_{2}^{g}(p)$};
    \draw[draw=aubinPurple,line width=.65pt,
      -{Stealth[length=4pt,width=3pt]},shorten >=2pt]
      (4.78,2.93) .. controls (5.35,3.60) and (6.30,3.93) .. (ball-label.west);
    \fill[aubinPurple] (p) circle[radius=2.7pt];
    \node[aubin label,text=aubinPurple,anchor=north,
      fill=white,inner sep=.6pt] at (3.84,2.48) {$p$};
  \end{scope}

  \draw[aubin map] (8.02,3.60) -- (9.30,3.60);
  \node[aubin label,anchor=south] at (8.66,3.78)
    {$\Phi_{p,1}$};

  % The plane is in front of the lower hemisphere, as in the sketch.
  \begin{scope}[shift={(12.30,3.27)}]
    \def\sphereR{1.92}
    \def\equatorry{.72}
    % Construct opposite sides from the same two vectors.
    \def\planeux{6.10}
    \def\planeuy{.20}
    \def\planevx{1.90}
    \def\planevy{2.20}
    \def\planepatch{
      (-4.00,-1.20) -- ++(\planeux,\planeuy)
      -- ++(\planevx,\planevy) -- ++(-\planeux,-\planeuy) -- cycle}
    \def\upperhemisphere{
      (-\sphereR,0) arc[start angle=180,end angle=0,radius=\sphereR]
      arc[start angle=0,end angle=-180,x radius=\sphereR,y radius=\equatorry]
      -- cycle}
    \pgfmathsetmacro{\viewElevation}{asin(\equatorry/\sphereR)}

    \path[fill=white] (0,0) circle[radius=\sphereR];
    \draw[aubin outline,fill=white] \planepatch;
    % Omit the rear border where it is covered by the upper hemisphere.
    \path[fill=white] \upperhemisphere;

    % The axis is dashed inside the sphere and solid beyond its poles.
    \draw[draw=black!40,line width=.45pt,dash pattern=on 2pt off 2.5pt]
      (0,-\sphereR) -- (0,\sphereR);
    % Halve the extensions beyond the north and south poles.
    \draw[draw=black,line width=.65pt] (0,-2.195) -- (0,-\sphereR);
    \draw[draw=black,line width=.65pt,-{Stealth[length=4pt,width=3pt]}]
      (0,\sphereR) -- (0,2.365);

    % Three red latitude spheres, with dashed rear arcs.
    \foreach \latitude in {22,46,70}{
      \pgfmathsetmacro{\rx}{\sphereR*cos(\latitude)}
      \pgfmathsetmacro{\ry}{\rx*sin(\viewElevation)}
      \pgfmathsetmacro{\cy}{\sphereR*sin(\latitude)*cos(\viewElevation)}
      \draw[draw=aubinRed!75,line width=.5pt,dash pattern=on 3pt off 2.5pt]
        (-\rx,\cy) arc[start angle=180,end angle=0,x radius=\rx,y radius=\ry];
      \draw[draw=aubinRed,line width=.70pt]
        (-\rx,\cy) arc[start angle=180,end angle=360,x radius=\rx,y radius=\ry];
    }

    % Green arcs covered by the foreground plane remain dashed.
    \foreach \latitude in {-22,-46,-70}{
      \pgfmathsetmacro{\rx}{\sphereR*cos(\latitude)}
      \pgfmathsetmacro{\ry}{\rx*sin(\viewElevation)}
      \pgfmathsetmacro{\cy}{\sphereR*sin(\latitude)*cos(\viewElevation)}
      \draw[draw=aubinGreen!75,line width=.5pt,dash pattern=on 3pt off 2.5pt]
        (0,\cy) ellipse[x radius=\rx,y radius=\ry];
      \begin{scope}[even odd rule,overlay]
        \clip (-5,-3) rectangle (5,3.5) \planepatch;
        \draw[draw=aubinGreen,line width=.70pt]
          (-\rx,\cy) arc[start angle=180,end angle=360,x radius=\rx,y radius=\ry];
      \end{scope}
    }

    \draw[draw=aubinPurple,line width=.95pt,dash pattern=on 4pt off 3pt]
      (-\sphereR,0) arc[start angle=180,end angle=0,x radius=\sphereR,y radius=\equatorry];
    \draw[draw=aubinPurple,line width=1.20pt]
      (-\sphereR,0) arc[start angle=180,end angle=360,x radius=\sphereR,y radius=\equatorry];
    % The side contour is hidden where the plane passes in front of it.
    \draw[draw=black!65,line width=.65pt,dash pattern=on 4pt off 3pt]
      (\sphereR,0) arc[start angle=0,end angle=-180,radius=\sphereR];
    \begin{scope}[even odd rule,overlay]
      \clip (-5,-3) rectangle (5,3.5) \planepatch;
      \draw[aubin outline]
        (\sphereR,0) arc[start angle=0,end angle=-180,radius=\sphereR];
    \end{scope}
    \draw[aubin outline]
      (-\sphereR,0) arc[start angle=180,end angle=0,radius=\sphereR];
    \draw[aubin outline] (-4.00,-1.20) -- ++(\planeux,\planeuy);
    \fill[aubinPurple] (0,\sphereR) circle[radius=2.7pt];
  \end{scope}
  \node[aubin label,anchor=north west] at (14.50,2.00) {$T_pZ$};
  \node[aubin label,anchor=east] at (12.24,5.585) {$\mathbb R$};
  \node[aubin label,text=aubinPurple,anchor=south west]
    at (12.42,5.39) {$\Phi_{p,1}(p)$};
\end{tikzpicture}%
}
\caption{Here $\lambda = 1$ and $\bar k =0$. The map $\Phi_{p,1}:Z\to\mathbb S^n\subset T_pZ\oplus\mathbb R$ maps $p$ to $(0,1)$. The map $\Phi_{p,1}$ maps each geodesic sphere in $B_2^g(p)$ centered at $p$ to a geodesic sphere centered at north pole $\Phi_{p,1}(p) = (0,1)$ (see red circles) and each geodesic sphere centered at $p$ outside $B_2^g(p)$ to a geodesic sphere below the equator (see green circles).}
\label{fig:balancing-construction}
\end{figure}
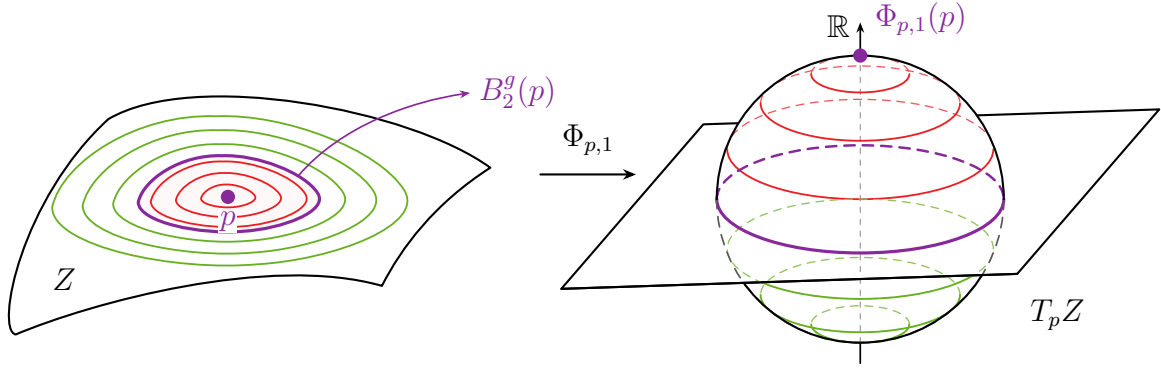

We collect several technical properties of the construction above when applied to the model space $\mathbb{M}^n_{\bar k}$ itself, these will be used later and proved in Appendix~\ref{app: prop of f i} and \ref{app: prop of f ii}.

{\begin{lemma}\label{lem: prop of f}
\begin{enumerate}[\upshape (i)]
    \item Let $\Phi^{\bar k}_{0,\lambda}$ and $f_{0,\lambda}$ denote the maps $\Phi_{p,\lambda}$ and $f_{p,\lambda}$ for $\mathbb{M}^n_{\bar k}$ at $p=0$, respectively. Then we have $(\Phi^{\bar k}_{0,\lambda})^*\sigma=f_{0,\lambda}(r)^2\delta^{\bar k}$.
    \item For every $p \in Z$, $\lambda > 0$, and $C^2$ embedded oriented hypersurface $\Sigma \subset Z$, the function $f = f_{p,\lambda}$ satisfies the following inequality in the weak sense:
    \begin{align*}
    2\Delta_{\Sigma}\log f &\leq -(n-3)|\nabla_\Sigma\log f|^2 - (n-1)(-\bar k+f^2)\nonumber -(n-1)(\partial_{\nu_\Sigma}\log f)^2 \\&\quad -2(n-1)H_{\Sigma}\,\partial_{\nu_\Sigma}\log f.
    \end{align*}
\end{enumerate}
\end{lemma}}

Since $\Psi_p$ is $1$-Lipschitz and preserves the radial coordinate, for a $C^1$ hypersurface $\Sigma^{n-1}$, Rademacher's theorem applied to $\Psi_p|_\Sigma$ gives differentiability of $\Psi_p|_\Sigma$ at $\cH_g^{n-1}$-almost every $x\in\Sigma$. Thus, since $\Phi^{\bar k}_{0,\lambda}(\Psi_p(x))=\Phi(x)$, the chain rule gives, for every $e\in T_x\Sigma$ and for $\cH^{n-1}_g$-almost every $x \in \Sigma$, $|d_\Sigma\Phi_x(e)|=\bigl|d(\Phi^{\bar k}_{0,\lambda})_{\Psi_p(x)}(d_\Sigma\Psi_p)_x(e)\bigr|=f(x)\bigl|(d_\Sigma\Psi_p)_x(e)\bigr|_{\delta^{\bar k}}
\le f(x)|e|_g$. Summing over an orthonormal basis $e_1,\ldots,e_{n-1}$ of $T_x\Sigma$, we obtain the following bound on the Hilbert–Schmidt norm:
\begin{equation}\label{eq:dPhi}
|d_\Sigma\Phi|^2 \le(n-1)f^2\ \text{for}\ \cH^{n-1}_g\text{-almost every point of}\ \Sigma.
\end{equation}

The proof of the following lemma is in Appendix~\ref{appendix:balancing}.

\begin{lemma}\label{lem:balance} 
Let $\mu$ be a compactly supported probability measure on $Z$ with $\mu(\{y\}) = 0$ for any $y\in Z$. There exist $p\in Z$ and $\lambda>0$ such that $\int_Z\Phi_{p,\lambda}(x)\,d\mu(x)=0$.
\end{lemma}
% Figure 3: support of the measure and its schematic spherical image.
% The balancing identity is stated in the preceding lemma.
% Complete figure: drawing, sizing, caption, and label.
% Loaded from main.tex; shared styles are in figures/aubin-styles.tex.
\begin{figure}[!htbp]
\centering
\setlength{\abovecaptionskip}{4pt}
\resizebox{.88\linewidth}{!}{%
\begin{tikzpicture}[aubin base]
  \ifdefined\aubinFullPage
    \path[use as bounding box] (0,1.25) rectangle (16,5.40);
  \fi

  % Align the panels vertically without changing their proportions.
  \begin{scope}[yshift=.65cm]
  \draw[aubin outline]
    (1.00,.81) .. controls (.72,.87) and (1.07,2.62) .. (2.24,3.69)
    .. controls (3.38,4.36) and (6.28,3.82) .. (7.35,3.04)
    .. controls (6.60,2.56) and (6.16,1.97) .. (5.90,1.46)
    .. controls (4.39,1.89) and (2.00,1.23) .. (1.00,.81) -- cycle;
  \node[aubin label] at (2.40,3.20) {$Z$};

  % A modestly wider, thicker support region, with a light fill for legibility.
  % Move the support and its base point together to the center of Z.
  \begin{scope}[shift={(-.10,.40)}]
  \begin{scope}[shift={(3.94,2.25)},xscale=1.15,yscale=1.75,shift={(-3.94,-2.25)}]
    \draw[aubin blue,fill=aubinBlue!7]
      (2.62,2.06) .. controls (3.12,2.51) and (4.29,2.47) .. (5.01,2.37)
      .. controls (5.29,2.34) and (5.48,2.16) .. (5.22,2.16)
      .. controls (4.37,2.25) and (3.22,2.13) .. (2.62,2.06) -- cycle;
  \end{scope}
  \node[aubin label,text=aubinBlue,anchor=south west,font=\large] at (5.08,2.63)
    {$\operatorname{spt}\mu$};
  \fill[aubinRed] (3.94,2.25) circle[radius=2pt];
  \node[aubin label,text=aubinRed,anchor=north] at (3.94,1.99) {$p$};
  \end{scope}
  \end{scope}

  \draw[aubin map] (8.02,3.60) -- (9.30,3.60);
  \node[aubin label,anchor=south] at (8.66,3.78) {$\Phi_{p,\lambda}$};

  % Sphere with a blue cap-shaped image of the support.
  \begin{scope}[yshift=-.40cm]
  \coordinate (c) at (12.30,3.67);
  \def\sphereradius{1.92}
  \def\equatorradius{.72}
  \def\supportimageedge{
      (10.62,4.66)
      .. controls (10.81,4.55) and (10.96,4.90) .. (11.27,4.99)
      .. controls (11.58,5.36) and (12.22,5.48) .. (12.66,5.28)
      .. controls (12.88,5.05) and (13.12,5.14) .. (13.36,4.96)
      .. controls (13.57,4.76) and (13.76,4.73) .. (13.91,4.90)}
  % Let the sphere outline determine the bounds of the clipped fill.
  \begin{scope}[overlay]
    \clip (c) circle[radius=\sphereradius];
    \path[fill=aubinBlue!22] \supportimageedge
      -- (14.35,6.08) -- (10.20,6.08) -- cycle;
    \draw[aubin blue] \supportimageedge;
  \end{scope}
  \draw[aubin outline] (c) circle[radius=\sphereradius];
  \draw[aubin hidden]
    (10.38,3.67) arc[start angle=180,end angle=0,x radius=\sphereradius,y radius=\equatorradius];
  \draw[aubin outline]
    (10.38,3.67) arc[start angle=180,end angle=360,x radius=\sphereradius,y radius=\equatorradius];
  % Mark the north pole at the exact top of the schematic sphere.
  \coordinate (northpole) at (12.30,{3.67+\sphereradius});
  \fill[aubinRed] (northpole) circle[radius=2.7pt];
  \node[aubin label,text=aubinRed,inner sep=1pt] (pole-label)
    at (12.15,4.78) {$\Phi_{p,\lambda}(p)$};
  \node[aubin label,text=aubinBlue,anchor=east] at (11.08,5.30)
    {$\Phi_{p,\lambda}(\operatorname{spt}\mu)$};
  \end{scope}

\end{tikzpicture}%
}
\caption{The support of $\mu$ in $Z$ and its image under $\Phi_{p,\lambda}$, as in Lemma~\ref{lem:balance}.}
\label{fig:measure-support}
\end{figure}
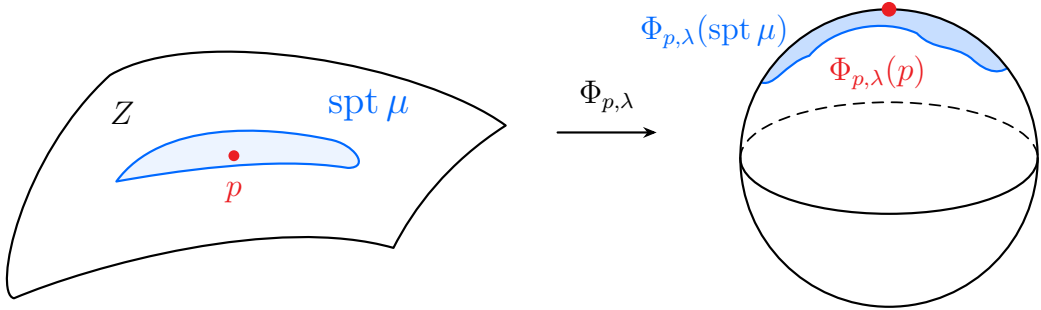

\subsection{A lower bound for $\mathcal{Q}$}
We prove a preliminary lower bound for $\mathcal{Q}$ in terms of the function $f$ from the previous subsection. Here $v$ will be compactly supported and Lipschitz. First, we rewrite $\mathcal{Q}$ as follows:
\begin{equation*}
\begin{aligned}
    \mathcal{Q}(v) &= \int \left((q-2)\left |\nabla_\Sigma v - \frac{n-3}{2}v\nabla_\Sigma \log f\right|^2 + (n-1)(H_\Sigma^2 + \bar k)v^2 \right.\\
    &\qquad + \left.(n-3)(q-2)\langle \nabla_\Sigma {v}, v\nabla_\Sigma\log f\rangle - \frac{(n-3)^2(q-2)}{4}v^2|\nabla_\Sigma\log f|^2\right).
\end{aligned}
\end{equation*}

Next, integrate by parts with $\Delta_\Sigma \log f$ and use {Lemma~\ref{lem: prop of f}(ii)} (the weak inequality can be tested against $v^2$, which is only Lipschitz, by using nonnegative {$C^1_c$} approximations in charts) to get
\begin{equation*}
\begin{aligned}
    \int 4\langle \nabla_\Sigma {v}, v\nabla_\Sigma\log f\rangle &\geq \int \left[ (n-3)|\nabla_\Sigma \log f|^2 + (n-1)(-\bar k + f^2)\right. \\
    &\qquad \left. + (n-1)(\partial_{\nu_\Sigma}\log f)^2 + 2(n-1)H_\Sigma\partial_{\nu_\Sigma}\log f \right]v^2
\end{aligned}
\end{equation*}

A direct algebraic manipulation of the last two displayed inequalities together with $q-2 = 4/(n-3)$ yields
\begin{equation}\label{eq:square}
\mathcal{Q}(v)\ge\int_\Sigma(n-1)v^2\left(f^2+(H_\Sigma+\partial_{\nu_\Sigma}\log f)^2\right)+(q-2)\left|\nabla_\Sigma v
-\frac{n-3}{2}v\nabla_\Sigma\log f\right|^2.
\end{equation}

Since all the coefficients in \eqref{eq:square} are bounded on a neighborhood of \(\overline U\), given \(v\in W^{1,2}_0(U)\), choose {\(v_j\in C_c^1(U)\)} converging strongly to \(v\) in \(W^{1,2}(U)\). Extending these functions by zero outside \(U\), we may apply \eqref{eq:square} to each \(v_j\). Passing to the limit proves \eqref{eq:square} for the zero extension of \(v\).

\section{Proof of Theorem~\ref{thm:smooth CH}}

Assume $n\ge4$, we prove the nonsmooth version of \Cref{thm:smooth CH} under a lower bound.

\begin{theorem}\label{thm:main}
%Let $(Z,d)$ be an non-collapsing Alexandrov space with no Alexandrov boundary and curvature bounded from above by $\bar k \leq 0$ and from below by $\underline k$. 
We have
\begin{equation}\label{eq:critical}
\ip{Z}{d}(v) \geq \ipm{\bar k}(v) \ \text{for every} \ v\geq 0.
\end{equation}
Moreover, if $\Omega \subset Z$ is a finite perimeter set with positive and finite volume and $P_{g}(\Omega) = \ipm{\bar k}(\vol_g(\Omega))$, then $(\Omega,d)$ is isometric to a geodesic ball in $\mathbb M^n_{\bar k}$.
\end{theorem}
\begin{remark}
In \cite[Theorem~1.1]{OpenAI2026CAT}, OpenAI showed the sharp Euclidean isoperimetric inequality for integral fillings in proper $\mathrm{CAT}(0)$ spaces, without a lower curvature bound. The equality case is not established there. This has \eqref{eq:critical} as a direct corollary when $\bar k=0$.
\end{remark}

By \cite[Proposition~8.1 and Theorem~1.3]{ADNRdiff}, to obtain Theorem~\ref{thm:main} and also Theorem~\ref{thm:smooth CH}, it suffices to prove the proposition below.

For $r\geq 0$, we denote by $V_{\bar k}(r)$ the volume of $B_r^{\delta^{\bar k}}$ in $(\mathbb M^n_{\bar k}, \delta^{\bar k})$. For {$V > 0$}, put $R_V:=V_{\bar k}^{-1}(V)$ as the radius of the ball of volume $V$ and $H_{\bar k {,V}}:=s_{\bar k}'(R_V)/s_{\bar k}(R_V)$ as its mean curvature. {Once $V$ is fixed, we abbreviate $R := R_V$ and $H_{\bar k} := H_{\bar k{,V}}$.}

\begin{proposition}\label{prop: equality case+first failure point}
Assume that, for some $V>0$, it holds
\begin{equation}\label{assump ineq}
\ip{Z}{d}(v) \geq \ipm{\bar k}(v), \ \text{for every} \ v \in (0,V].
\end{equation}
If $\Omega \subset Z$ is a finite perimeter set such that $\vol_g(\Omega) = V$ and $P_{g}(\Omega) = \ipm{\bar k}(V)$, then $(\Omega,d)$ is isometric to $(B_R^{\delta^{\bar k}},\delta^{\bar k})$.
\end{proposition}

Before proving the proposition, we need the following lemma. 

\begin{lemma}\label{lem:H}
Under the assumptions of Proposition~\ref{prop: equality case+first failure point}, we have $0<H_{\partial \Omega}\le H_{\bar k}$.
\end{lemma}
\begin{proof}
By \cite[Theorem~1.4]{APP}, we take a bounded open representative of $\Omega$ which, by \cite[Corollaries~3.7--3.8]{Morgan2003}, has singular set of Hausdorff dimension at most $n-8$ and then $P_g(\Omega)=\cH_g^{n-1}(\mathrm{reg}(\partial\Omega))$. On the regular set of $\partial\Omega$, the mean curvature $H_{\partial\Omega}$ is constant.
Fix $o\in Z$. By \eqref{eq:hessdistance}, the scalar function $r_o^2/2$ has weak Hessian at least $g$. By a standard mollification argument, cf. Lemma~\ref{lem:restriction}(ii), on a compact neighborhood of $\overline\Omega$, we obtain {$C^{3,\alpha}$} functions $(r_o^2/2)_\varepsilon$ converging to $r_o^2/2$ in $C^1$, and numbers $\delta_\varepsilon\to0$ such that $\operatorname{Hess}_g(r_o^2/2)_\varepsilon\ge(1-\delta_\varepsilon)g$ on that neighborhood. Consequently,
\[
\operatorname{div}_{\mathrm{reg}(\partial \Omega)}\nabla(r_o^2/2)_\varepsilon
=\operatorname{tr}_{T{\mathrm{reg}(\partial \Omega)}}\operatorname{Hess}_g(r_o^2/2)_\varepsilon
\ge (n-1)(1-\delta_\varepsilon),
\
\Delta_g(r_o^2/2)_\varepsilon\ge n(1-\delta_\varepsilon).
\]

Multiplying $\nabla( r_o^2/2)_\varepsilon$ by a compactly supported cutoff that equals one near $\overline\Omega$ produces a $C_c^1$ vector field $X$. Therefore, the first variation of area and Gauss--Green formula give
\begin{align*}
H_{\partial \Omega}\int_{\partial \Omega}\langle\nabla(r_o^2/2)_\varepsilon,\nu_{\partial \Omega}\rangle\,d\cH^{n-1}_g
&\ge (n-1)(1-\delta_\varepsilon)P_g(\Omega),\\
\int_{\partial \Omega}\langle\nabla(r_o^2/2)_\varepsilon,\nu_{\partial \Omega}\rangle\,d\cH^{n-1}_g
&=\int_\Omega\Delta_g(r_o^2/2)_\varepsilon\,d\cH^{n}_g
\ge n(1-\delta_\varepsilon)V.
\end{align*}
Taking the limit as $\varepsilon\to 0$ yields
\begin{equation*}\label{eq:positiveflux}
    H_{\partial \Omega}\int_{\partial \Omega}\langle\nabla(r_o^2/2),\nu_{\partial \Omega}\rangle\,d\cH^{n-1}_g\ge (n-1)P_g(\Omega) 
\end{equation*}
and
\begin{equation*}
\int_{\partial \Omega}\langle\nabla(r_o^2/2),\nu_{\partial \Omega}\rangle\,d\cH^{n-1}_g\ge nV.
\end{equation*}
Since $P_g(\Omega)>0$ and $V>0$, we get $H_{\partial \Omega}>0$.

Take $x\in \mathrm{reg}(\partial \Omega)$, extend the outward normal from $B_{r_x}^g(x)\cap\mathrm{reg}(\partial \Omega)$ to a $C^1$ ambient vector field $\tilde \nu_{\partial \Omega}$ on the same set. Let $\phi$ be a nonnegative cutoff with positive integral on $\partial \Omega$ and supported in $B^g_{r_x}(x)$. Define $X\in {\mathfrak{X}}_c^1(Z)$ by $X:= -\phi \tilde \nu_{\partial \Omega}/ c|\tilde \nu_{\partial \Omega}|_g$ where $c := \int_{\mathrm{reg}(\partial\Omega)}\phi d\cH^{n-1}_g {> 0}$, then we have 
\[
\int_{\mathrm{reg}(\partial \Omega)}\langle X,\nu_{\partial \Omega}\rangle\,d\cH^{n-1}_g = {-\frac{1}{c}}\int_{\mathrm{reg}(\partial \Omega)\cap \spt\phi}\phi \,d\cH^{n-1}_g = -1.
\]
Let $X_t$ be its flow. Volume variation, the first variation of the area, {and our assumption,} give
\begin{equation}\label{eq:inwardvariation}
V_t := \vol_g(X_t\Omega)=V-t+o(t),
\qquad
P_g(X_t\Omega)=\ipm{\bar k}(V)- {(n-1)}H_{\partial\Omega}t+o(t).
\end{equation}
Therefore, \eqref{assump ineq} applies to $X_t\Omega$ for small $t$, hence $$
P_g(X_t\Omega)\ge \ip{Z}{d}(V_t)
\ge \ipm{\bar k}(V_t).$$

{Since $\ipm{\bar k}'(V) = (n-1)H_{\bar k}$, Taylor expansion, and \eqref{eq:inwardvariation}, give}
\[
\ipm{\bar k}(V_t)=\ipm{\bar k}(V)- {(n-1)}H_{\bar k}t+o(t).
\]
Combining the last two displayed equations with \eqref{eq:inwardvariation} and dividing by ${(n-1)}t>0$, we obtain
$-H_{\partial \Omega}+o(1)\ge-H_{\bar k}+o(1)$.
Letting $t\downarrow0$ concludes the proof.
\end{proof}

\begin{proof}[Proof of Proposition~\ref{prop: equality case+first failure point}]
Take the bounded open representative of $\Omega$ as in the proof of Lemma~\ref{lem:H}. %We record here that $\Omega$ must be connected. Indeed, if $\Omega = \Omega_1\sqcup \Omega_2$ with $\cH^n_d(\Omega_i)=v_i>0$, by the subadditivity of $\ipm{\bar k}$ {\cite[Lemma~3.3]{ADNRdiff}}, we obtain
%$$ P_d(\Omega) = P_d(\Omega_1)+P_d(\Omega_2) \overset{\eqref{assump ineq}}{\geq} \ipm{\bar k}(v_1)+\ipm{\bar k}(v_2) > \ipm{\bar k}(v) = P_d(\Omega).$$
%This gives the desired contradiction and then connectedness of $\Omega$.
Since $\Omega$ is an isoperimetric region, {\cite[Corollaries~3.7--3.8]{Morgan2003} gives that} $\mathrm{Sing}(\partial \Omega)$ is of Hausdorff dimension at most $n-8$. Therefore, by a standard capacity argument (using in particular $\cH^{n-2}_g(\mathrm {Sing}(\partial \Omega)) = 0$), we can apply \eqref{eq:sob} with $v=1$ to obtain $
(n-1)(H_{\partial \Omega}^2+\bar k) P_g(\Omega) \ge c_nP_g(\Omega)^{(n-3)/(n-1)}$. Since $P_g(\Omega) = \ipm{\bar k}({V}) = \omega_{n-1} s_{\bar k}(R)^{n-1}$, {$H_{\bar k}=s_{\bar k}'(R)/s_{\bar k}(R)$, and  $(s_{\bar k}')^2+\bar k s_{\bar k}^2=1$}, we obtain
\begin{equation}\label{eq: new}
H_{\partial \Omega}^2+\bar k\ge\left(\frac{\omega_{n-1}}{P_d(\Omega)}\right)^{2/(n-1)}
=s_{\bar k}(R)^{-2}={H_{\bar k}^2}+\bar k.
\end{equation}
Lemma~\ref{lem:H} gives $0<H_{\partial \Omega}\le {H_{\bar k}}$, so the last inequality forces $H_{\partial \Omega}={H_{\bar k}}$. Therefore the constant function attains equality in the Michael-Simon inequality, i.e., 
\begin{equation}\label{eq:extremal}
    \mathcal{Q}(1)=(n-1)(H_{\bar k}^2+\bar k)P_g(\Omega)=c_nP_g(\Omega)^{(n-3)/(n-1)}.
\end{equation}

\medskip

\emph{Step 1: The balancing map.}
Since $\Omega$ is isoperimetric, it is bounded by \cite[Theorem~1.4]{APP}. Then the probability measure $(\cH^{n-1}_g\res \partial \Omega)/ P_g(\Omega)$ is compactly supported and we can apply Lemma~\ref{lem:balance}. Let $p$, $\lambda$, and map $\Phi$, with scalar factor $f$ be as in the lemma. Let $w_1,\ldots,w_{n+1}$ be its coordinates in an orthonormal basis of $T_pZ\oplus\R$. Then
\begin{equation}\label{eqn: zero mean}
    \sum_jw_j^2=1,
\qquad\int_{\partial \Omega} w_j\,d\cH^{n-1}_g=0,
\qquad |w_j|\le1.
\end{equation}

For $|t|<1/2$, the function $1+tw_j$ is positive. Taylor expanding $t\mapsto (1+t\omega_j)^{{q}}$ and $y\mapsto y^{2/q}$, and since the zero mean \eqref{eqn: zero mean} removes the linear term, we get 
\begin{equation}\label{eqn: taylor}
    \|1+tw_j\|_{L^q(\partial \Omega)}^2 =P_g(\Omega)^{2/q}+(q-1)P_g(\Omega)^{2/q-1}t^2 \|w_j\|_{L^2(\partial \Omega)}^2+o(t^2).
\end{equation}
Moreover, a direct computation using {that $H_{\partial \Omega}$ is constant and} \eqref{eqn: zero mean} gives $\mathcal{Q}(1+tw_j)=\mathcal{Q}(1)+t^2\mathcal{Q}(w_j)$. Using this, \eqref{eq:extremal}, \eqref{eqn: taylor}, we apply \eqref{eq:sob} to $1+tw_j$, divide by $t^2$, and let $t\to0$, to obtain
\[
\mathcal{Q}(w_j)\ge(q-1)c_nP_g(\Omega)^{2/q-1}
\|w_j\|^2_{L^2(\partial \Omega)}. 
\]
Recalling that $c_nP_g(\Omega)^{2/q-1}=(n-1)(H_{\bar k}^2+\bar k)$ by \eqref{eq:extremal} and plugging in the definition of $\mathcal{Q}$ imply the Poincar\`e inequality, i.e.,
\begin{equation}\label{eqn: poincare w_j}
\|\nabla_{\partial \Omega} w_j\|_{L^2(\partial\Omega)}^2
\ge (n-1)(H_{\bar k}^2+\bar k)\|w_j\|_{L^2(\partial\Omega)}^2.
\end{equation}

The sum of the squared gradients of the coordinates $w_1,\ldots, w_{n+1}$ is $|d_{\partial \Omega}\Phi|^2$. Using this, the last displayed inequality, and the first assertion in \eqref{eqn: zero mean}, gives
\begin{equation}\label{eq:spectral}
\|d_{\partial \Omega} \Phi\|_{L^2(\partial\Omega)}^2
\ge (n-1)(H_{\bar k}^2+\bar k)P_g(\Omega) \overset{\eqref{eq:extremal}}{=} \mathcal Q(1).
\end{equation}

\medskip

\emph{Step 2: $\Omega$ is a geodesic ball.}
Use \eqref{eq:square} with $v=1$ for this same balanced map $\Phi$, $H_{\partial \Omega} = H_{\bar k}$, and $(q-2)(n-3)^2/4 = n-3$, to get
\[
\begin{split}
    \mathcal{Q}(1)&\ge (n-1)\int_{\partial \Omega} f^2\,d\cH^{n-1}_g +(n-1)\int_{\partial \Omega}(H_{\bar k}+{\partial_{\nu_{\partial \Omega}}}\log f)^2\,d\cH^{n-1}_g \\
    &\quad +(n-3)\int_{\partial \Omega}|\nabla_{\partial \Omega}\log f|^2\,d\cH^{n-1}_g.
\end{split}
\]
Putting \eqref{eq:dPhi} and \eqref{eq:spectral} together yields
\begin{align*}
 0\geq (n-1)\int_{\partial \Omega}(H_{\bar k}+{\partial_{\nu_{\partial \Omega}}}\log f)^2\,d\cH^{n-1}_g +(n-3)\int_{\partial \Omega}|\nabla_{\partial \Omega}\log f|^2\,d\cH^{n-1}_g.
\end{align*}
Since both nonnegative integrals must be zero (because $n\geq 4$), the two integrands are continuous, and the singular set of $\partial \Omega$ is of zero $\cH^{n-1}_g$-measure, we have
\begin{equation}\label{eq:zeros}
\nabla_{\partial \Omega}\log f=0 \ \text{ and }\ 
 {\partial_{\nu_{\partial \Omega}}}\log f=-H_{\bar k}
\quad\text{at every point of }{\mathrm{reg}(\partial \Omega)}.
\end{equation}

Equality in the inequalities above and \eqref{eq:spectral} also give
\begin{equation}\label{eq:farea}
\int_{\partial \Omega} f^2\,d\cH^{n-1}_g=(H_{\bar k}^2+\bar k)P_g(\Omega).
\end{equation}

By definition, the ambient gradient of $\log f$ is zero at $p$, while its normal derivative in \eqref{eq:zeros} is nonzero, and then $p\notin{\mathrm{reg}(\partial \Omega)}$. On $\mathrm{reg}(\partial \Omega)$, we have
\begin{equation}\label{eqn: ambient gradient}
    \nabla\log f=(\partial_r\log f)(r_p)\nabla r_p
\ \text{ and } \ (\partial_r\log f)(r_p)<0.    
\end{equation}
Equation \eqref{eq:zeros} says that the same vector is $-H_{\bar k}\nu_{\partial \Omega}$, by taking the norms, and using $|\nabla r_p|=1$, we obtain $(\partial_r\log f)(r_p)=-H_{\bar k}$ which implies
\begin{equation}\label{eq:radialnormal}
    \nu_{\partial \Omega}=\nabla r_p
\quad\text{on }{\mathrm{reg}(\partial \Omega)}.
\end{equation}
The second assertion in \eqref{eqn: ambient gradient} fixes the outward orientation. 

Equation \eqref{eq:radialnormal} gives that $r_p$ is constant on each connected component of $\mathrm{reg}(\partial \Omega)$. Next, we show that this constant is the same on all components. 

Rearranging the terms in the definition of $f$ and $w_{n+1}$, we see that $w_{n+1} = C_1 f - C_2$ where both constants only depend on $\bar k$ and $\lambda$, namely $C_1 = 2\lambda/(\lambda^2 - \bar k)$ and $C_2 = (\lambda^2 + \bar k)/(\lambda^2 - \bar k)$. Then, by \eqref{eq:zeros}, $\nabla_{\partial \Omega}w_{n+1} = 0$ and, by \eqref{eqn: poincare w_j}, $w_{n+1} \equiv 0$. This gives $z(r_p) = 1$ by definition of $\Phi$. Therefore, $r_p$ is constant on $\mathrm{reg}(\partial \Omega)$. 

Since $s_{\bar k}(r_p) = 2s_{\bar k}(r_p/2)s_{\bar k}'(r_p/2)$, by $z(r_p) =1$ and the definitions of $z$ and $f$, we get $f(r_p) = 1/s_{\bar k}(r_p)$. By \eqref{eq:farea} and \eqref{eq: new}, we get $r_p = R$ and since $\Omega$ is open and bounded, we must have that $\Omega = B_R^g(p)$.

\medskip

\emph{Step 3: An \emph{almost} everywhere isometry. } Let $\Psi=\Psi_p$ be the map defined in
\eqref{eq:radial-comparison-map}, and set
$F:=\Psi^{-1}|_{B_R^{\delta^{\bar k}}}$.
This is a locally Lipschitz homeomorphism {onto $\Omega = B_R^g(p)$}, and $F$ is $1$-Lipschitz and then differentiable almost everywhere. Let $d^{\bar k}$ be the distance induced by $\delta^{\bar k}$. Applying
$d^{\bar k}(x,x')\le d(F(x),F(x'))$ to $x'=\exp_x^{\mathbb{M}^n_{\bar k}}(tv)$, dividing by $|t|$, and passing to the limit, at every differentiability point, we get
\[
|dF_x(v)|_g\ge |v|_{\delta^{\bar k}}, \quad \forall v\in T_x\mathbb{M}_{\bar k}^n.
\]
Therefore, its Jacobian satisfies $J_F\ge1$ almost everywhere. Since $F$ is injective {and $F(B_R^{\delta^{\bar k}}) = \Omega$ by Step 2}, the Riemannian area formula gives
\begin{equation*}
    \vol_{\delta^{\bar k}}(B_R^{\delta^{\bar k}}) = \vol_g(\Omega) =\int_{B_R^{\delta^{\bar k}}}J_F\,d\vol_{\delta^{\bar k}}.%, \qquad J_F\ge1\ \text{almost everywhere}.
\end{equation*}
We then conclude $J_F=1$ almost everywhere and, in particular,
\begin{equation}\label{eq:recoveryderivative}
|dF_x(v)|_g=|v|_{\delta^{\bar k}}
\quad\text{for all }v\text{ at almost every }x\in B_R^{\delta^{\bar k}}.
\end{equation}

\smallskip\noindent\emph{Step 3: The isometry. }
We show \eqref{eq:recoveryderivative} for all $x\in B_R^{\delta^{\bar k}}$ through a simple approximation argument. Fix distinct $x,y\in B_R^{\delta^{\bar k}}$, and let $\gamma:[0,L]\to B_R^{\delta^{\bar k}}$ be the unit-speed model geodesic from $x$ to $y$.

Choose a parallel orthonormal frame $E_1,\ldots,E_{n-1}$ in the normal bundle of $\gamma$. For sufficiently small {$\tau>0$}, the tubular map
\[
\Theta(t,z)=\exp_{\gamma(t)}^{\mathbb{M}^n_{\bar k}}
\left(\sum_{i=1}^{n-1}z_iE_i(t)\right),
\qquad (t,z)\in(-\tau,L+\tau)\times B_\tau^{\mathbb R^{n-1}},
\]
is a smooth diffeomorphism onto its image, and
$\Theta([0,L]\times B_\tau^{\mathbb R^{n-1}})\subset B_R^{\delta^{\bar k}}$.
The set of nondifferentiability points of $F$ has zero model volume and its preimage under the tubular diffeomorphism has zero Lebesgue measure. By Fubini, for almost every $z\in B_\tau^{\mathbb R^{n-1}}$, the curve $t\mapsto\Theta(t,z)$ meets this set in a set of one-dimensional measure zero.

The composition $t\mapsto F(\Theta(t,z))$ is absolutely continuous, because $F$ is locally Lipschitz and each nearby curve has compact image. For almost every $z$, by \eqref{eq:recoveryderivative}, we get $\left|\partial_tF(\Theta(t,z))\right|_g
 = |\partial_t\Theta(t,z)|_{\delta^{\bar k}}$, for almost every $t\in[0,L]$. This in turn implies $d\bigl(F(\Theta(0,z)),F(\Theta(L,z))\bigr) \le \int_0^L|\partial_t\Theta(t,z)|_{\delta^{\bar k}}\,dt$. Choose admissible $z_k\to0$. The endpoints converge to $x$ and $y$, $F$ is continuous, and the nearby model curves converge to $\gamma$ in $C^1$. Passing to the limit yields $d(F(x),F(y))\le L=d^{\bar k}(x,y)$. The reverse inequality follows directly from the $1$-Lipschitz property of $\Psi = F^{-1}$. Thus $F$ is an isometry and this concludes the proof.
\end{proof}

%\newpage

\appendix

\section{Proof of the technical lemmas}
\label{appendix:restriction}

\subsection{Proof of item (i) in Lemma~\ref{lem:restriction}}\label{app:lem restriction item i}
Let $\eta\in C^1_c(\Sigma)$ be nonnegative and assume first that
$\spt\eta$ is contained in one ambient chart $U$.
Choose $V\Subset U$ containing $\spt\eta$, and let
$ u _\varepsilon:= u *\rho_\varepsilon$ where $\rho_\varepsilon$ is a nonnegative standard mollifier. For sufficiently small $\varepsilon$, $ u _\varepsilon$ is a $C^{3,\alpha}$ function defined on a neighborhood of $\overline V$, and $u _\varepsilon\longrightarrow u$ in $C^1(\overline V)$.

In these coordinates, $\mathrm{Hess}_g u \leq B$ becomes $D^2 u \le B+\Gamma d u $, where
$(\Gamma d u )_{ij}:=\sum_k\Gamma^k_{ij}\partial_k u $ and $\Gamma_{ij}^k$ denote the Christoffel symbols.
Since weak derivatives commute with convolution and
$\rho_\varepsilon\ge0$, we obtain
$D^2 u _\varepsilon
=(D^2 u )*\rho_\varepsilon
\le(B+\Gamma d u )*\rho_\varepsilon$ and using that $\operatorname{Hess}_g u _\varepsilon
=D^2 u _\varepsilon-\Gamma d u _\varepsilon$,
we then get
\[
\operatorname{Hess}_g u _\varepsilon\le B+E_\varepsilon,\ \text{and}\
E_\varepsilon:=(B+\Gamma d u )*\rho_\varepsilon
-B-\Gamma d u _\varepsilon.
\]
Since $B$, $\Gamma$, and $d u $ are continuous, $(B+\Gamma du)*\rho_\varepsilon$ converges uniformly to $B+\Gamma d u $ on $\overline V$.
Together with $ u _\varepsilon\to u $ in
$C^1(\overline V)$, this gives $E_\varepsilon\to0$ uniformly.
By uniform ellipticity of $g$ on $\overline V$, there exist
$\delta_\varepsilon\ge0$, with $\delta_\varepsilon\to0$, such that $|E_\varepsilon(w,w)|\le\delta_\varepsilon g_p(w,w)$ for every $p\in\overline V$ and $w\in T_pZ$. Consequently,
\begin{equation}\label{eq:scalarapprox-appendix}
\operatorname{Hess}_g u _\varepsilon
\le B+\delta_\varepsilon g \text{ on }V.
\end{equation}

For the $C^{3,\alpha}$-function $u _\varepsilon$, since $\nabla^\Sigma u_\varepsilon = \nabla^g u_\varepsilon - (\partial_{\nu_\Sigma}u_\varepsilon) {\nu_{\Sigma}}$, a direct computation gives $\Delta_\Sigma u _\varepsilon =\operatorname{tr}_{T\Sigma} \operatorname{Hess}_g u _\varepsilon -(n-1)H_\Sigma\,\partial_{\nu_\Sigma}u_\varepsilon$.
This, taking the tangential trace in
\eqref{eq:scalarapprox-appendix}, multiplying by $\eta$,
and integrating by parts (recall that $\spt \eta\Subset \Sigma$), yield
\[
-\int_\Sigma
\langle\nabla_\Sigma u _\varepsilon,
\nabla_\Sigma\eta\rangle\,d\cH_g^{n-1}\le\int_\Sigma\eta\bigl(
\operatorname{tr}_{T\Sigma}B+(n-1)\delta_\varepsilon
-(n-1)H_\Sigma\,\partial_{\nu_\Sigma}u_\varepsilon
\bigr)\,d\cH_g^{n-1}.
\]
Letting $\varepsilon\to0$ and using the $C^1$ convergence proves the lemma for test functions compactly supported in one chart.

For a general nonnegative test function $\eta\in C^1_c(\Sigma)$, choose
a finite nonnegative partition of unity $\{\chi_i\}_i$ on its support, apply
the last displayed inequality to each $\chi_i\eta$, and sum on $i$ to conclude. Replacing $ u $ and $B$ by $-u$ and $-B$ gives the reversed
inequality.

\subsection{Proof of item (ii) in Lemma~\ref{lem:restriction}}\label{app:lem restriction item ii}
All the components of the proof are already in the proof of item (i), we only have to show \eqref{eq:scalarapprox-appendix} works well with the partition of unity. Fix a compact set $K\subset Z$ {and a relatively compact neighborhood $U$ of $K$}. Choose finitely many ambient charts $U_i$ covering {$\overline{U}$} and a nonnegative {$C^{3,\alpha}$} partition of unity $\chi_i$, with $\spt\chi_i\Subset U_i$ and $\sum_i\chi_i=1$ on a neighborhood of {$\overline{U}$}. For sufficiently small $\varepsilon$, let $ u _{\varepsilon,i}$ be the mollification constructed above and then \eqref{eq:scalarapprox-appendix} holds for $u_{\varepsilon,i}$ with $\delta_{\varepsilon ,i}\to 0$. Define $ u _\varepsilon:=\sum_i\chi_i u _{\varepsilon,i}${, which is $C^{3,\alpha}$ in the fixed atlas}. Since $\sum_i\chi_i=1$, we have $\sum_i d\chi_i=0$ and $\sum_i\operatorname{Hess}_g\chi_i=0$ on this neighborhood. Using these identities, the product rule, and setting $f_{\varepsilon,i}:=u_{\varepsilon,i}-u$, we obtain
\[
    \operatorname{Hess}_g u _\varepsilon
    = \sum_i\chi_i \operatorname{Hess}_g u _{\varepsilon,i} +\sum_i\left( d\chi_i\otimes df _{\varepsilon,i}+df_{\varepsilon,i}\otimes d\chi_i +f_{\varepsilon,i} \operatorname{Hess}_g\chi_i\right).
\]
{Since $u\in C^{1,1}_{\mathrm{loc}}$, standard mollification gives
$f_{\varepsilon,i}\to0$ locally in both $C^1$ and $W^{2,p}$ for every
$p\geq1$. In particular, the $C^1$ convergence implies that
the second sum converges uniformly to zero on $K$.} As above, uniform ellipticity gives numbers $a_\varepsilon\ge0$, with $a_\varepsilon\to0$, such that this sum is bounded above by $a_\varepsilon g$. Since the partition is nonnegative, the first sum satisfies
\[
\sum_i\chi_i\operatorname{Hess}_g u _{\varepsilon,i}
\le\sum_i\chi_i(B+\delta_{\varepsilon,i}g)
\le B+\left(\max_i\delta_{\varepsilon,i}\right)g.
\]
Consequently,
$\operatorname{Hess}_g u _\varepsilon
\le B+\delta_\varepsilon g$ on $K$, with
$\delta_\varepsilon:=
\max_i\delta_{\varepsilon,i}+a_\varepsilon\to0$.
Moreover, $u _\varepsilon- u  =\sum_i\chi_i f_{\varepsilon,i}$ near {$\overline{U}$},  {so the finite and fixed partition of unity gives $u_{\varepsilon} \to u$ in $C^1(\overline{U})$ and $W^{2,p}(U)$.} This finishes the proof.

\subsection{Proof of item (i) in Lemma~\ref{lem: prop of f}}\label{app: prop of f i}
We first prove {$|d\Phi^{\bar k}_{0,\lambda}w|_{\sigma} = f_{0,\lambda}(r)|w|_{\delta^{\bar k}}$} by hand. Recall that $\Phi^{\bar k}_{0,\lambda}$ is given by the inverse stereographic projection with radial coordinate $z=\lambda s_{\bar k}(r/2)/s_{\bar k}'(r/2)$. Since $s_{\bar k}''=-\bar k s_{\bar k}$ and $(s_{\bar k}')^2+\bar k s_{\bar k}^2=1$, we have $\partial_r z=\lambda/(2s_{\bar k}'(r/2)^2)$, and hence $2\partial_r z/(1+z^2)=f_{0,\lambda}(r)$. Moreover, using the identity $s_{\bar k}(r)=2s_{\bar k}(r/2)s_{\bar k}'(r/2)$, we obtain $2z/(1+z^2)=f_{0,\lambda}(r)s_{\bar k}(r)$. Since the round metric on the sphere in inverse stereographic coordinates is $4(1+z^2)^{-2}(dz^2+z^2\sigma)$, it follows that $(\Phi^{\bar k}_{0,\lambda})^*\sigma =4(1+z^2)^{-2}((\partial_r z)^2dr^2+z^2\sigma) =f_{0,\lambda}(r)^2(dr^2+s_{\bar k}(r)^2\sigma) =f_{0,\lambda}(r)^2\delta^{\bar k}$.

\subsection{Proof of item (ii) in Lemma~\ref{lem: prop of f}}\label{app: prop of f ii}
Now, we prove the inequality in the lemma. Straightforward computations by directly taking the derivatives give
\begin{align}
    \partial_r\log f &=-\frac{(\lambda^2-\bar k)s_{\bar k}(r)}{2\bigl(1+(\lambda^2-\bar k)s_{\bar k}(r/2)^2\bigr)}<0, \text{ for }r >0 ,\label{eq:radialfirst}\\
    -(\partial_r\log f)\frac{s_{\bar k}'(r)}{s_{\bar k}(r)} &=\frac12\bigl(-\bar k+f(r)^2+(\partial_r\log f)^2\bigr), \label{eq:radial}\\
    \partial_r^2\log f &=(\partial_r\log f)^2 +(\partial_r\log f)\frac{s_{\bar k}'(r)}{s_{\bar k}(r)}\label{eq:radialsecond}.
\end{align}

Next, we prove the following inequality in the weak sense
\begin{equation}\label{eq:hessu}
        \Hess_g\log f \le d\log f\otimes d\log f -\frac12(-\bar k+f^2+|\nabla\log f|^2)g.
\end{equation}
Recall that $\mathrm{Hess}_g \log f = \partial_r^2\log f dr\otimes dr + \partial_r\log f\mathrm{Hess}_gr$, then, by $\partial_r \log f <0$ (see \eqref{eq:radialfirst}) and \eqref{eq:hessr}, we obtain $\mathrm{Hess}_g \log f \leq (\partial_r^2\log f - \partial_r \log f s_{\bar k}'/s_{\bar k}) dr\otimes dr + (\partial_r \log f s_{\bar k}'/s_{\bar k}) g$. Putting this, \eqref{eq:radialsecond}, and \eqref{eq:radial} together, we derive \eqref{eq:hessu}. The same inequality holds {weakly} across $p$.
Indeed, {$f$ is a smooth function of $r_p^2$ near $p$, so $\log f\in C^{1,1}_{\mathrm{loc}}(Z)$ and $\nabla\log f(p)=0$.}

Recall that $\Delta_\Sigma \log f = \mathrm{div}_\Sigma(\nabla^g\log f) - \mathrm{div}_\Sigma((\partial_{\nu_\Sigma}\log f)\nu_\Sigma)$ and then $\Delta_\Sigma \log f = \mathrm{tr}_{T\Sigma}\mathrm{Hess}_g\log f - (n-1)H_\Sigma\partial_{\nu_\Sigma}\log f$ almost everywhere on $\Sigma$. We apply Lemma~\ref{lem:restriction}(i) to $u=\log f$ and $B$ being the right-hand side of \eqref{eq:hessu} to obtain (in the weak sense)
$$ \Delta_\Sigma \log f\leq |\nabla_\Sigma\log f|^2 - \frac{n-1}{2} (-\bar k+f^2 + |\nabla \log f|^2) - (n-1)H_\Sigma\partial_{\nu_\Sigma}\log f. $$
This inequality and $|\nabla\log f|^2=|\nabla_\Sigma\log f|^2+(\partial_{\nu_\Sigma}\log f)^2$ conclude the proof.

\subsection{Proof of Lemma~\ref{lem:balance}}\label{appendix:balancing}
Fix $o\in Z$, and choose $R>\max_{x\in\spt\mu}d(o,x)$. 
For $\xi\in T_oZ$ with $|\xi|\le {R}$, put $p=\exp_o\xi$. Parallel transport along the unique radial geodesic from $p$ to $o$ gives an isometry $P_{p,o} :T_pZ\to T_oZ$ and we define
\[
F(\xi,\lambda)=
\int_Z (P_{p, o},\mathrm{id}_{\R})
\Phi_{p,\lambda}(x)\,d\mu(x)
\in T_oZ\oplus\R.
\]
We write $F=(F_1,F_2)$ according to these two factors.

\medskip

\emph{Step 1: Continuity of $F$.}
The unique geodesics $\gamma$ depend continuously on their endpoints by the $\mathrm{CAT}$ property, see \cite[Ch II, Prop 1.4(1), pg 160]{BH}. Because the parameter interval is $[0,1]$, the constant speed is $|\dot\gamma_k|_g=d(p_k,o_k)$, which is uniformly bounded. On compact sets, we have $g_z(w,w)\ge a|w|_{\delta_{euc}}^2$ for some $a>0$ independent of $z$. Consequently, $|\dot\gamma_k|_{\delta_{euc}}\le a^{-1/2}|\dot\gamma_k|_g$.
So the velocities are uniformly bounded on compact sets as well.

In coordinates, $\ddot\gamma_k^i
=-\Gamma^i_{j\ell}(\gamma_k)\dot\gamma_k^j\dot\gamma_k^\ell$ and the Christoffel symbols are continuous. Together with the velocity bounds, this gives a uniform acceleration bound. Hence the velocities are equicontinuous $|\dot\gamma_k(t)-\dot\gamma_k(s)|\le C|t-s|$. Arzel\`a--Ascoli gives a subsequence for which both positions and velocities converge uniformly. Since the paths already converge to $\gamma$, we must have $\dot\gamma_k\to\dot\gamma$ and the entire sequence converges in $C^1$.

For distinct limiting endpoints, this also gives continuity of the initial unit direction:
\[
\theta_{p_k}(o_k)
=\frac{\dot\gamma_k(0)}{d(p_k,o_k)}
\longrightarrow
\frac{\dot\gamma(0)}{d(p,o)}
=\theta_p(o).
\]

A parallel vector field $W$ along $\gamma$ satisfies $\dot W=-A(t)W$ with {$A(t)^i_\ell = \Gamma^i_{j\ell}(\gamma(t))\dot\gamma^j(t)$}. By the convergences above, $A_k\to A$ uniformly. Solutions of linear ODEs depend continuously on the coefficients and on their initial vectors. Thus the parallel transport maps vary continuously with the endpoints.

At $x=p$, the direction $\theta_p(x)$ is not defined, but the norm of the first component of $\Phi_{p,\lambda}$ (cf. \eqref{eq:Phi}) is at most $0\le 2\lambda \frac{s_{\bar k}(r_p(x)/2)}{s_{\bar k}'(r_p(x)/2)}\le\lambda r_p(x)$, which tends to zero as $x\to p$, uniformly when $\lambda$ lies in a compact interval. The last component of $\Phi_{p,\lambda}$ tends to one. Thus the transported integrand is jointly continuous at $p=x$. Since $\mu$ is compactly supported, we have that $F$ is continuous.

\medskip

\emph{Step 2: The signs on the faces of the cylinder.}
On the side $|\xi|=R$, the point $p$ lies outside $\spt\mu$. Radial parallel transport gives $P_{p, o}(R\nabla r_o(p))=\xi$. For $x\in\spt\mu$, convexity of $r_o$ along the unit-speed geodesic from $p$ to $x$ gives
\[
\frac{\langle\xi,P_{p, o}\theta_p(x)\rangle}{R} = \langle\nabla r_o(p),\theta_p(x)\rangle
\le\frac{r_o(x)-r_o(p)}{d(p,x)}
=\frac{r_o(x)-R}{d(p,x)}<0.
\]

The scalar multiplying $\theta_p$ in the first component of $\Phi_{p,\lambda}$ (see \eqref{eq:Phi}) is strictly positive here. Integrating, we obtain
\begin{equation}\label{eq:balanceside}
\langle\xi,F_1(\xi,\lambda)\rangle<0
\quad\text{when }|\xi|=R, \lambda>0.
\end{equation}

For every $|\xi|\le R$ and $x\in\spt\mu$, we have $r_p(x)<2R$. Choose $\lambda_0>0$ so that $\lambda_0\frac{s_{\bar k}(R)}{s_{\bar k}'(R)}<1$. Then the last component of $\Phi_{p,\lambda}$ (cf. \eqref{eq:Phi}) is strictly positive, uniformly over this compact set. Hence
\begin{equation}\label{eq:balancebottom}
F_2(\xi,\lambda_0)>0
\quad\text{for }|\xi|\le R.
\end{equation}

By the assumption that $\mu$ is nonatomic and compactly supported, and continuity of measures, we obtain that $\sup_{p\in Z}\mu(B_\tau^g(p))\longrightarrow0$ as $\tau\downarrow0$.

Choose $\tau>0$ such that this supremum is less than $1/4$, and choose $\lambda_1>\lambda_0$ with
$\lambda_1\frac{s_{\bar k}(\tau/2)}{s_{\bar k}'(\tau/2)}\ge\sqrt3$.
If $x\notin B_\tau^g(p)$, the last component of $\Phi_{p,\lambda_1}$ in \eqref{eq:Phi} is at most $-1/2$; inside the ball it is at most one. Therefore, since $\mu$ is a probability measure, we have
\begin{equation}\label{eq:balancetop}
F_2(\xi,\lambda_1) \leq \mu(B_\tau^g(p)) - \frac12\mu(B_\tau^g(p)^c) 
\le-\frac12+\frac32\mu(B_\tau^g(p))
\le-\frac18<0.
\end{equation}

\smallskip\noindent\emph{Step 3: A zero in the interior.}
Consider the compact convex cylinder
\[
C=\{\xi\in T_oZ:|\xi|\le R\}\times[\lambda_0,\lambda_1]
\subset T_oZ\oplus\R.
\]
The map $\zeta\mapsto\operatorname{proj}_C(\zeta+F(\zeta))$ is a continuous self-map of $C$. By Brouwer's fixed point theorem it has a fixed point $\zeta$. The characterization of Euclidean projection gives
\begin{equation}\label{eq:projection}
\langle F(\zeta),y-\zeta\rangle\le0,
\qquad\text{for every }y\in C,
\end{equation}
indeed, let $a:=\operatorname{proj}_C(\zeta+F(\zeta))$ and $z:= \zeta+F(\zeta)$, the segment $a+t(y-a)$ lies in $C$, and $a$ minimizes distance to $z$, we obtain \eqref{eq:projection} by deriving $
|z-a-t(y-a)|^2$ with respect to $t>0$.

If $\zeta$ lies on the lower face, an upward displacement (i.e., choosing $y > \zeta$ sufficiently close to $\zeta$) and \eqref{eq:projection} contradict \eqref{eq:balancebottom}. Analogously, if $\zeta$ lies on the upper face, a downward displacement and \eqref{eq:projection} contradict \eqref{eq:balancetop}. If $\zeta = (\zeta^\prime, \zeta_0)$ and $|\zeta^\prime|=R$, a small displacement in the direction $(-\zeta^\prime,0)$ and \eqref{eq:projection} contradict \eqref{eq:balanceside}. Therefore $\zeta$ is an interior point. 

Applying \eqref{eq:projection} with $y=\zeta \pm \epsilon\zeta$ for small enough $\epsilon>0$ forces $F(\zeta)=0$. This finishes the proof after applying the inverse of {$(P_{p,o}^{-1},\mathrm{id}_{\R})$} to the definition of $F$. 

\section{Proof of the Michael-Simon inequality}\label{app: MS}

\subsection{Proof of Theorem~\ref{prop:sob}}

Recall that $g$ is $C^{1,\alpha}$, and $H_\Sigma$ is locally bounded. 

\medskip

\emph{Step 1: A non-optimal inequality. }
Fix $U\Subset\Sigma$ and $\varepsilon>0$. At a point, we take a small chart $\varphi:B_r\subset \R^{n-1}\to \Sigma$ so that $(1-\tau)\mathrm{id}\le(g_{ij})\le(1+\tau)\mathrm{id}$. The sharp Euclidean Sobolev inequality of Aubin and Talenti \cite{Aubin1976,Talenti1976} together with standard Riemannian comparisons implies, for $v$ as in the theorem supported in $\varphi(B_r)$,
\begin{equation}\label{eqn:talenti aubin}
    c_n\left(\frac{1-\tau}{1+\tau}\right)^{(n-1)/2}
\|v\|_{L^q(\Sigma)}^2
\le (q-2)\|\nabla_\Sigma v\|_{L^2(\Sigma)}^2.
\end{equation}
Indeed, in the chart, $|\nabla_\Sigma v|^2=g^{ij}\partial_i v\,\partial_j v$ and $d\cH_g^{n-1}=\sqrt{\det g}\,dx$.
Inverting the metric bounds gives $(g^{ij})\ge(1+\tau)^{-1}\mathrm{id}$. Every eigenvalue of $g$ is at least $1-\tau$, so $\sqrt{\det g}\ge(1-\tau)^{(n-1)/2}$. Therefore 
$$\|\nabla_\Sigma v\|_{L^2(\Sigma)}^2 \ge\frac{(1-\tau)^{(n-1)/2}}{1+\tau}\|\nabla(v\circ \varphi)\|_{L^2(dx)}^2.$$ 
The bound $\sqrt{\det g}\le(1+\tau)^{(n-1)/2}$ implies $\| v\|_{L^q(\Sigma)}^q\le(1+\tau)^{(n-1)/2}\| v\circ\varphi\|_{L^q(dx)}^q$, taking the power $2/q$ leads to
\[
\|v\circ\varphi\|_{L^q(dx)}^2\ge(1+\tau)^{-(n-1)/q}\|v\|_{L^q(\Sigma)}^2.
\]
The last two displayed inequalities {and Aubin and Talenti \cite{Aubin1976,Talenti1976}} imply \eqref{eqn:talenti aubin}.

Cover $\overline U$ by finitely many such charts, and choose a partition of unity $\chi_i$ such that $\sum_i\chi_i^2=1$ on a neighborhood of $\overline U$. By Minkowski's or Young's inequality in $L^{q/2}$, we obtain $\|v\|_q^2 = \left\|\sum_i(\chi_i v)^2\right\|_{q/2} \le\sum_i\|\chi_i v\|_q^2$.
Moreover, $\sum_i\chi_i\nabla_\Sigma\chi_i=0$, so
\[
\sum_i|\nabla_\Sigma(\chi_i v)|^2
=|\nabla_\Sigma v|^2
+v^2\sum_i|\nabla_\Sigma\chi_i|^2.
\]
Applying \eqref{eqn:talenti aubin} to $\chi_i v$, summing in $i$, and choosing $\tau$ sufficiently small depending on $\varepsilon$, give
\begin{equation}\label{eq:localsob}
(c_n-\varepsilon)\|v\|_{L^q(\Sigma)}^2
\le (q-2)\|\nabla_\Sigma v\|_{L^2(\Sigma)}^2
+C_{\varepsilon,U}\| v\|_{L^2(\Sigma)}^2,
\end{equation}
for any $v$ as in the statement of the theorem. Standard density results for Sobolev spaces extend \eqref{eq:localsob} to $W^{1,2}_0(U)$ functions $v$, cf. \cite[Section~9.4, Remark~18]{Brezis2011}, applied in local charts.

\medskip
\emph{Step 2: Existence of a minimizer when below the threshold $c_n$.}
Set
\[
c(U)=\inf\{\mathcal{Q}(v):v\in W^{1,2}_0(U),\ \|v\|_{L^q(\Sigma)}=1\}.
\]
Equation \eqref{eq:square} gives $c(U)\ge0$. Suppose by contradiction $$c(U)<c_n,$$ and choose a minimizing sequence $v_k$. By H\"older's inequality, we have $\|v_k\|^2_{L^2(U)}
\le \bigl(\cH_g^{n-1}(U)\bigr)^{2/(n-1)}$ which{, by definition of $\mathcal{Q}$,} implies
\[
(q-2)\|\nabla_\Sigma v_k\|^2_{L^2(U)}
\le\mathcal{Q}(v_k)
-(n-1)\bar k\bigl(\cH_g^{n-1}(U)\bigr)^{2/(n-1)}.
\]
Thus the sequence is bounded in $W^{1,2}_0(U)$. After passing to a subsequence,
\[
v_k\rightharpoonup {v_0}\text{ in }W^{1,2}_0(U),\qquad
v_k\to {v_0}\text{ in }L^2(U)\text{ and almost everywhere},
\]
where the $L^2$ convergence follows since $W_0^{1,2}$ is compactly embedded into $L^2$ by Rellich's theorem, cf. \cite[Theorem~9.16 and Section~9.4, Remark~20]{Brezis2011}, applied to zero extensions in local charts. By Fatou's lemma, $v_0\in L^q(U)$ {and $\|v_0\|_q \leq 1$.}

Let $t=\| {v_0}\|_{L^q(U)}^q$, so $0\le t\le1$. The Br\'ezis--Lieb lemma \cite{BL} gives $\|v_k- {v_0}\|^q_{L^q(U)} = 1-t+o(1)$. This, strong $L^2$ convergence{, and weak $W^{1,2}_0$ convergence,} give
\[
\mathcal{Q}(v_k)=\mathcal{Q}( {v_0})
+(q-2)\|\nabla_\Sigma(v_k- {v_0})\|^2_{L^2(U)}
+o(1).
\]

Apply \eqref{eq:localsob} to $v_k- {v_0}$ and let $k\to\infty$ to obtain 
$$ c(U) \geq \mathcal{Q}(v_0) + (c_n-\varepsilon)(1-t)^{2/q}. $$
Let $\varepsilon\downarrow0$ and recall the definition of {$\mathcal{Q}$ and} $t$ to get 
\begin{equation}\label{eq:concentration}
c(U)\ge c(U)t^{2/q}+c_n(1-t)^{2/q}.
\end{equation}

If $c(U)=0$, this gives $t=1$. If $0<c(U)<c_n$ and $t<1$, by \eqref{eq:concentration} and since $2/q <1$, we get
\[
c(U) > c(U)\bigl(t^{2/q}+(1-t)^{2/q}\bigr)\ge c(U).
\]
This is a contradiction and hence $t=1$. Weak lower semicontinuity gives $\mathcal{Q}( {v_0})\le c(U)$, and, since $\| v_0\|_{L^q(U)} = t^{{1/q}} = 1$, we have $\mathcal{Q}( {v_0})= c(U)$.

\medskip

\emph{Step 3: Some variational inequalities.} For any bounded Lipschitz function $\zeta$ {with $\int_U |v_0|^q \zeta d\mathcal{H}^{n-1}_g = 0$}, consider $ {v_0}(1+s\zeta)$. For small $s$, the factor $1+s\zeta$ is positive, and Taylor expansion together with $\|v_0\|_{L^q(U)} = 1$ give
\begin{equation}\label{taylor}
    \| {v_0}(1+s\zeta)\|^q_{L^q(U)} 
    =1 +qs\int_U| {v_0}|^q\zeta\,d\cH_g^{n-1} + \frac{q(q-1)}2s^2
\int_U| {v_0}|^q\zeta^2\,d\cH_g^{n-1}+o(s^2).   
\end{equation}
The quotient ${\mathcal{Q}( {v_0}(1+s\zeta))}/{\| {v_0}(1+s\zeta)\|_q^{2}}$ has a minimum at $s=0$. The quotient rule at $s=0$, {\eqref{taylor} and $\int_U |v_0|^q \zeta d\mathcal{H}^{n-1}_g = 0$} then shows
\begin{equation}\label{eq:minfirst}
\begin{aligned}
    \int_U\Bigl((q-2)\langle\nabla_\Sigma {v_0},\nabla_\Sigma( {v_0}\zeta)\rangle +(n-1)(H_\Sigma^2+\bar k) v_0^2\zeta\Bigr) = c(U)\int_U| {v_0}|^q\zeta = 0.
\end{aligned}
\end{equation}

By \eqref{taylor} and the Taylor expansion of $s\mapsto (1+s)^{2/q}$, we obtain
\begin{equation*}\label{taylor Lq}
  \| {v_0}(1+s\zeta)\|_{L^q(U)}^2 = 1 + (q-1)s^2\int_U |v_0|^q\zeta^2\,d\cH^{n-1}_g + o(s^2).  
\end{equation*}
By definition of $\mathcal{Q}$ {and \eqref{eq:minfirst}}, we obtain 
\begin{equation*}\label{expansion Q}
    \mathcal{Q}(v_0(1+s\zeta)) = \mathcal{Q}(v_0) + s^2\mathcal{Q}(v_0\zeta).
\end{equation*}
Using {the last two displayed equations} and Taylor expanding in the quotient, we get
$$ \frac{\mathcal{Q}( {v_0}(1+s\zeta))}{\| {v_0}(1+s\zeta)\|_q^{2}} = \mathcal{Q}(v_0) + s^2\left( \mathcal{Q}(v_0\zeta) - (q-1)\mathcal{Q}(v_0)\int_U |v_0|^q\zeta^2\,d\cH^{n-1}_g \right) +o(s^2).$$

Since $v_0$ is a minimum for the quotient and $\|v_0\|_{L^q(U)} = 1$ {for $s$ small}, we obtain
\begin{equation*}\label{eq:minsecond}
\mathcal{Q}( {v_0}\zeta)
\ge(q-1)c(U)\int_U| {v_0}|^q\zeta^2\,d\cH_g^{n-1}.
\end{equation*}
Applying \eqref{eq:minfirst} with $\zeta^2$ and expanding the gradients, we also obtain the exact identity
\begin{equation*}\label{eq:minproduct}
\mathcal{Q}( {v_0}\zeta)
=c(U)\int_U| {v_0}|^q\zeta^2\,d\cH_g^{n-1}
+(q-2)\int_U v_0^2|\nabla_\Sigma\zeta|^2\,d\cH_g^{n-1}.
\end{equation*}
{Subtracting the last two displayed inequalities we obtain
\begin{equation}\label{eq: new B8}
\int_U v_0^2|\nabla_\Sigma\zeta|^2\,d\cH_g^{n-1} \geq c(U)\int_U| {v_0}|^q\zeta^2\,d\cH_g^{n-1}.
\end{equation}
All these calculations are valid when $c(U)=0$ as well.}

\medskip

\emph{Step 4: Using $c_n >c(U)$ we prove $v_0 = 0$.} The computations in this step are entirely analogous to Steps 1 and 2 in the proof of \Cref{prop: equality case+first failure point}.

Assume for the moment that $c_n>c(U)$ (as in Step 2) and $\Sigma$ is connected and that $\Sigma\setminus\overline U$ contains a nonempty open set. 

The measure $| {v_0}|^q\,d\cH_g^{n-1}$, extended by zero and viewed as a measure on $Z$, is a nonatomic probability measure with compact support. Then we can apply Lemma~\ref{lem:balance}, and let $w_1,\ldots,w_{n+1}$ be the spherical coordinates of the resulting map $\Phi$. They satisfy
\[
\sum_{j=1}^{n+1}w_j^2=1,
\qquad\int_U| {v_0}|^q w_j\,d\cH_g^{n-1}=0.
\]
Each $w_j$ is bounded and Lipschitz on a neighborhood of $\overline U$, so multiplication by it preserves $W^{1,2}_0(U)$.

Take $\zeta=w_j$ in \eqref{eq: new B8}, sum over $j$, and recall $\sum_jw_j^2=1$, to get
\[
c(U)\le\int_U {v_0}^2|d_\Sigma\Phi|^2\,d\cH_g^{n-1}.
\]
Combining this with \eqref{eq:dPhi} and \eqref{eq:square} yields
\begin{align*}
c(U)
&\le\int_U {v_0}^2|d_\Sigma\Phi|^2
\overset{\eqref{eq:dPhi}}{\le}(n-1)\int_U f^2 {v_0}^2\\
\overset{\eqref{eq:square}}&{\le}\mathcal{Q}( {v_0})
-\int_U\left( (n-1)(H_\Sigma+{\partial_{\nu_{\Sigma}}}\log f)^2 v_0^2  +(q-2)\left|\nabla_\Sigma {v_0}-\frac{n-3}{2} {v_0}\nabla_\Sigma\log f\right|^2\right)\\
&\le \mathcal{Q}(v_0) = c(U).
\end{align*}
Therefore $\nabla_\Sigma {v_0}
=\frac{n-3}{2} {v_0}\nabla_\Sigma\log f$ almost everywhere on $U$ and then $f^{-(n-3)/2} {v_0}$ has zero weak gradient on $U$. Since $v_0\in W^{1,2}_0(U)\cap L^q(U)$ and $f$ is $C^1$ and positive, we can extend $f^{-(n-3)/2}v_0$ to a function in $W^{1,2}_0(\Sigma)$ that is zero outside $\bar U$. Since $\Sigma$ is connected and $\Sigma\setminus\bar U$ contains an open set, $f^{-(n-3)/2}v_0$ is constantly equal to zero and, since $f>0$, this contradicts $\|v_0\|_{L^q(U)} = 1$.

We have then proved $c(U)\ge c_n$ whenever $\Sigma$ is connected and $\Sigma\setminus\overline U$ contains a nonempty open set.

\medskip

\emph{Step 5: General $\Sigma$ and $U$.}
For any Lipschitz function $v$ with compact support properly contained in a connected component $\Sigma_0$ of $\Sigma$, {\eqref{eq:sob}} holds for $v$.

If $\Sigma_0$ is compact and {a Lipschitz function $\tilde v$ has} $\mathrm{spt}( \tilde v )= \Sigma_0$, choose a point $p\in\Sigma_0$ and a cutoff $\eta_\delta$ that vanishes in $B_\delta(p)$, equals one outside $B_{2\delta}(p)$, and has gradient bounded by $C_0/\delta$. Then
\[
\int_\Sigma|\nabla_\Sigma\eta_\delta|^2\,d\cH_g^{n-1} \leq \int_{B_{2\delta}^g(p)\setminus B_\delta^g(p)}|\nabla_\Sigma\eta_\delta|^2\,d\cH_g^{n-1} \le C_1\delta^{n-3}\longrightarrow0.
\]
This shows that, for compactly supported in $\Sigma_0$ Lipschitz functions $\tilde v$, the functions $\eta_\delta \tilde v$ converge to $\tilde v$ in $W^{1,2}$ and $L^q$. Each $\eta_\delta \tilde v$ has proper compact support in $\Sigma_0$, so passing to the limit proves {\eqref{eq:sob}} for $\tilde v$.

Finally, since a compact support meets only finitely many connected components, we apply the inequality to each restriction, sum, and use 
$\sum_i a_i^{2/q}\ge(\sum_i a_i)^{2/q}$ (since  $2/q<1$), to obtain {\eqref{eq:sob}} on their union. This finishes the proof of the theorem in the case of Lipschitz functions.

The $W^{1,2}_0(U)$ assertion follows from standard density results for Sobolev spaces and using \eqref{eq:localsob} for convergence in $L^q$.

\begingroup\small\raggedright
\bibliographystyle{amsplain}
\bibliography{9biblio}
\endgroup

\end{document}